\documentclass[12pt]{elsarticle}

\usepackage{amssymb}
\usepackage{amsmath}
\usepackage{bbm}
\usepackage{mathrsfs} 
\usepackage{adjustbox}
\usepackage{bm}
\usepackage{dsfont}
\usepackage{comment}
\usepackage{longtable}
\usepackage{multicol}
\usepackage{threeparttable}
\usepackage{tabularx}
\usepackage{algpseudocode}
\usepackage{algorithm}
\usepackage{xcolor}
\journal{Opeartions Research Letter}
\newtheorem{theorem}{Theorem}
\newtheorem{lemma}[theorem]{Lemma}
\newtheorem{corollary}{Corollary}
\newtheorem{proposition}{Proposition}
\newdefinition{assumption}{Assumption}
\newdefinition{definition}{Definition}
\newproof{proof}{Proof}
\newproof{pot}{Proof of Theorem \ref{thm2}}

\begin{document}

\begin{frontmatter}



\title{Convergence and Bound Computation for Chance Constrained Distributionally Robust Models using Sample Approximation}


\author[inst1]{Jiaqi Lei}

\affiliation[inst1]{organization={Department of Industrial Engineering and Management Sciences, Northwestern University},%Department and Organization
            %addressline={2145 Sheridan Rd}, 
            city={Evanston},
            postcode={60208}, 
            state={IL},
            country={USA}}

\author[inst1]{Sanjay Mehrotra}
%\author[inst1,inst2]{Author Three}

%\affiliation[inst2]{organization={Department Two},%Department and Organization
%            addressline={Address Two}, 
%            city={City Two},
%            postcode={22222}, 
%            state={State Two},
%            country={Country Two}}

\begin{abstract}
This paper considers a distributionally robust chance constraint model with a general ambiguity set. We show that a sample based approximation of this model converges under suitable sufficient conditions. We also show that upper and lower bounds on the optimal value of the model can be estimated statistically. Specific ambiguity sets are discussed as examples.
\end{abstract}




\end{frontmatter}


\section{Introduction}
A robust optimization model is typically thought of as a substitute for a chance constraint optimization model \cite{Ben98, Bertsimas04}. %since the framework can be motivated if the model parameters are normally distributed.  
A distributionally robust optimization (DRO) model assumes partial knowledge of the probability distribution specifying the data and it describes an optimization model where the distribution belongs to an ambiguity set ${\cal P}$ (see \cite{Rahimian19} for a review). Recently, it was observed that a chance constraint optimization model benefits from a DRO framework in improving the out-of-sample performance of satisfaction of the chance constraint \cite{WangLiMehrotra22}. The model in \cite{WangLiMehrotra22} assumed that a finitely supported nominal distribution is available when describing the ambiguity set. The finitely supported nominal distribution is not always available. For example, $\mathcal{P}$ may be specified through moment-based constraints \cite{Delage10,Mehrotra14} (see Section~\ref{sec:examples}). This paper, under suitable conditions, shows that using (sample based) discretization of the support on which the probability distributions in ${\cal P}$ are specified, the distributionally robust chance constraint model approximation converges. Moreover, it is possible to generate statistical estimates of the lower and upper bounds for the model. More specifically, the results are shown for the  robust chance constraint optimization model:
\begin{equation}\label{dro_model}
\begin{aligned}
v^*:=\min_{x\in\mathcal{X}} \; q(x):=\max_{\mathbb{P}\in \mathcal{P}}\; &\mathbb{E}_\mathbb{P}[F(x,\xi)]\\
\text{s.t.}\; &\text{Prob}_\mathbb{P}\left(G(x,\xi)\leq 0 \right)\geq 1-\theta, 
\\
%& \mathcal{P}=\{\mathcal{P}|\mathbb{E}[\mu_j]=\bar{\mu_j},\mathbb{E}[(\mu_j-\bar{\mu_j})^2]\leq \sigma_j^2\}
\end{aligned}
\end{equation}
where $\theta\in[0,1]$ is a safety level.

Note that the robustness in the chance constraint in \eqref{dro_model} is specified through the objective function. The model in \eqref{dro_model} captures the situation where the uncertainty in the objective also impacts the uncertainty in constraints. For example, the same data determines the expected reward and risk in a classical portfolio optimization model. The objective functions of the models can be formulated with respect to the expected utility over the choice of a distribution. The constraints include the requirements of the maximum risk level to be satisfied with a given probability \cite{Fabozzi07,Omidi17}. The results are proved under the following assumptions, and Assumption~\ref{ambiguity_assumption}.  
\begin{assumption}\label{basic_assumption}
(1.1) The support $\Xi\subseteq\mathbb{R}^d$ is compact, and $||\xi||\leq M_\Xi$ for any $\xi\in \Xi$, which also implies that $\Xi$ has the diameter $\max_{\xi_i,\xi_j\in\Xi}||\xi_i-\xi_j||\leq 2M_\Xi$. (1.2) $\mathcal{X}\subseteq\mathbb{R}^n$ is a nonempty set and ${\cal X}$ can be mixed-integer. (1.3) The solution set of \eqref{dro_model} is assumed to be nonempty and compact with finite optimal values; and a solution of the inner problem in \eqref{dro_model} exists. (1.4) For a given $\mathbb{P}\in\mathcal{P}$, we assume that the optimal value $v_\mathbb{P}(\theta):=\min_{x\in\mathcal{X}}\{\mathbb{E}_\mathbb{P}[F(x,\xi)]|\text{Prob}_\mathbb{P}\Big(G(x,\xi)\leq 0 \Big)\geq 1-\theta\}$ is locally Lipschitz continuous in $\theta$ with Lipschitz constant $\kappa^\theta$. See \cite[Theorem 1]{henrion2004perturbation} for detailed conditions on $F(x,\xi)$ and $G(x,\xi)$ ensuring that this assumption holds. (1.5) For all $x\in\mathcal{X}$ and $\xi\in\Xi$, we further assume that the functions $F(x, \xi)$, $G(x,\xi):\mathbb{R}^n \times \Xi \rightarrow \mathbb{R}^m$ are measurable and Lipschitz continuous in $\xi$ with Lipschitz constants $\kappa^F$, $\kappa^G$, respectively. (1.6) The functions $F(\cdot,\cdot)$ and $G(\cdot,\cdot)$ are assumed to be finite valued. (1.7) We also assume that the continuous distributions in $\mathbb{P}\in\mathcal{P}$ have a bounded probability density $f_\mathbb{P}(\cdot)$  with the bound $C^\mathcal{P}$\footnote{A probability density is bounded if for the random variable 
$Y$ with the probability density $f_\mathbb{P}(\cdot)$ and domain $[a,b]$, i.e.  $\text{Prob}_\mathbb{P}(a\leq Y\leq b)=\int_a^b f_\mathbb{P}(y)dy$, the probability density $f_\mathbb{P}(\cdot)$ satisfies  $f_\mathbb{P}(y)\leq C^\mathcal{P}$ for any $a\leq y\leq b$.}.  
\end{assumption}

\subsection{Literature Review}
In the current literature, the convergence results that ensure asymptotic consistency of the optimal value and establish a quantitative relationship between the sample size and error in optimal value are known only for the case where the distributional robustness is specified for the objective function  \cite{Bertsimas18,Liu19,Xu18}.
For example, Bertsekas and Gupta \cite{Bertsimas18} proposed a Robust Sample Average Approximation (SAA), which linked the SAA and \eqref{DRO} model as 
\begin{equation}\tag{DRO}\label{DRO}
   \min_{x\in\mathcal{X}} \max_{\mathbb{P}\in \mathcal{P}
   %(\xi_1,\dots,\xi_\omega)
   }
   \;\mathbb{E}_\mathbb{P}[F(x,\xi)],
\end{equation}
where $\mathcal{P}$ is a distributional uncertainty set with data-driven samples. The model retained tractability, asymptotic properties, and finite-sample performance guarantee of the SAA model. 
Liu, Pichler, and Xu \cite{Liu19} introduced the discrete approximation of \eqref{DRO}, 
where $\mathcal{P}$ is the discretization of the moment-based ambiguity set. A key result in their paper quantified the difference between the ambiguity set and its discretization. It also showed the convergence of the optimal solutions.
Xu, Liu, and Sun \cite{Xu18} solved \eqref{DRO} with $\mathcal{P}$ specified as a moment-based ambiguity set by reformulating the inner maximization problem using duality, and discretizing $\Xi$ through Monte Carlo sampling.  They also provided an upper bound for the optimal value of DRO using duality and a lower bound by a discrete approximation approach. 

Additionally, some papers discuss convergence results for models that include distributionally robust chance constraints but without the robustness in the objective function \cite{Jiang24,Zymler13,guo2015stability,sun2019convergence}. 
Jiang and Peng \cite{Jiang24} considered the models
with distributionally robust chance constraints \eqref{DRCP} as 
\begin{equation}\tag{DRCP}\label{DRCP}
 \min_{x\in\mathcal{X}}F(x),\; \text{s.t.}\inf_{\mathbb{P}\in\mathcal{P}}\text{Prob}_\mathbb{P}\left(G(x,\xi)\leq 0 \right)\geq 1-\theta.   
\end{equation}
Note that in this model, the objective function is deterministic. Using a discrete and finite support set, the authors investigated the discrete approximation of the model and provided convergence results for the optimal solutions. 
Zymler, Kuhn and Rustem \cite{Zymler13} considered \eqref{DRCP} using the moment-based ambiguity sets and approximating the constraints by worst-case conditional-value-at-risk (CVaR) constraints. 
%The approximation quality was controlled by a set of scaling parameters. 
Guo, Xu and Zhang \cite{guo2015stability} analyzed the convergence of the optimal value and the optimal solutions of \eqref{DRCP} when a generic approximation of the ambiguity set converges to the ambiguity set $\mathcal{P}$. Sun, Zhang and Chen \cite{sun2019convergence} proposed a DC approximation method to \eqref{DRCP} and gave a convergence analysis of the DC approximation.

Other papers consider the robustness in both the objective function and the constraints, but the distribution $\mathbb{P}$ are decided separately in the models \cite{Xie17,Zhou20}.
Xie and Ahmed \cite{Xie17} provided a distributionally robust chance-constrained model \eqref{DRCCM}
\begin{equation}\tag{DRCCM}\label{DRCCM}
   \min_{x\in\mathcal{X}} \max_{\mathbb{P}\in \mathcal{P}}\;\mathbb{E}_\mathbb{P}[F(x,\xi)], \; \text{s.t.} \inf_{\mathbb{P}\in\mathcal{P}}\text{Prob}_\mathbb{P}\left(G(x,\xi)\leq 0 \right)\geq 1-\theta 
\end{equation}
using the moment-based ambiguity set with an application to renewable energy in optimal power flow. To solve the model efficiently, they reformulated the constraint using duality and second-order cone programming. 
Zhou et al. \cite{Zhou20} proposed the \eqref{DRCCM} specified using the Wasserstein ambiguity set with an application to real-time dispatch. The CVaR approximation was applied to replace the robust chance constraints with deterministic linear constraints.

In contrast, the model in \eqref{dro_model} takes the robustness of the objective and constraints into consideration simultaneously. The worst-case distribution affects the objective and constraints simultaneously. As mentioned, this coupling of objective function and constraints is relevant for situations where the chance constraints and the objective function are driven from the same underlying data; for example, in portfolio optimization models.

\subsection{Contributions}
We study a sample-based approximation model and show convergence in the objective value of the approximated model to \eqref{dro_model} under a general condition on the distance between the ambiguity set and its discretization.   We provide examples of ambiguity sets that satisfy this assumption. We also suggest approaches for computing $100(1-\alpha)\%$ confidence interval of the optimal objective value of the approximated model. 

\section{Model Approximation and Known Technical Results}\label{subsec:stoch_model}
 Let $\mathcal{D}$ be the set of all distributions that can be specified on $\Xi$, and  $\mathcal{P}\subseteq \mathcal{D}$. Let $\Xi_{|\Omega|}:=\{\xi_1, \xi_2,\dots,\xi_{|\Omega|}\}$ 
 be  a discretization of $\Xi$ 
 %following a uniform distribution, 
 and $\xi_j\neq \xi_j$ if $i\neq j$. The elements of
 $\Xi_{|\Omega|}$ can be quantizers, or sampled to follow an independent identically distributed uniform distribution on  $\Xi$  (see Section~\ref{sec:known_results}).
 Let $\mathbb{P}_{|\Omega|}$ be a discrete probability distribution on the support $\Xi_{|\Omega|}$ and $\mathbb{P}_{|\Omega|}\in\mathcal{P}_{|\Omega|}  \subset {\cal P}$, where $\mathcal{P}_{|\Omega|}$ is the set of all discrete probability distributions in $\mathcal{P}$ that are supported on the set $\Xi_{|\Omega|}$. 

 The model in \eqref{dro_model} is approximated as:
\begin{equation}\label{dro_model_sampled}
\begin{aligned}
\hat{v}_{|\Omega|}:=\min_{x\in\mathcal{X}} q_{|\Omega|}(x):=\max_{\mathbb{P}_{|\Omega|}\in \mathcal{P}_{|\Omega|}} &\mathbb{E}_{\mathbb{P}_{|\Omega|} }[F(x,\xi)]\\%=\sum_{\omega=1}^{|\Omega|} p_\omega F(x,\xi_\omega)\\
\text{s.t.}\; & %\sum_{\omega=1}^{|\Omega|} p_\omega1_{(0, \infty)}G(x,\xi_\omega)\leq \theta
\text{Prob}_{\mathbb{P}_{|\Omega|}}\left(G(x,\xi)\leq  0 \right)\geq 1-\theta.\\
%& \mathcal{P}=\{\mathcal{P}|\mathbb{E}[\mu_j]=\bar{\mu_j},\mathbb{E}[(\mu_j-\bar{\mu_j})^2]\leq \sigma_j^2\}
\end{aligned}
\end{equation}
We assume that a solution of the inner problem in \eqref{dro_model_sampled} exists. This assumption can be ensured by a sufficiently large $|\Omega|$ for specific definitions of ${\cal P}$ has a non-empty interior (e.g., Assumption~\ref{def_slater}). The chance constraint in \eqref{dro_model_sampled} can be rewritten as:
\begin{equation}\label{specific_app_const}
   \begin{aligned}
&\text{Prob}_{\mathbb{P}_{|\Omega|}}\left(G(x,\xi)\leq 0 \right)
  = \mathbb{E}_{\mathbb{P}_{|\Omega|}}\left[\mathbbm{1}_{(-\infty,0]}\left(G(x, \xi)\right)\right] = \sum_{\omega=1}^{|\Omega|} p_\omega\mathbbm{1}_{(-\infty,0] }\left(G(x, \xi_\omega)\right)\geq 1-\theta,
   \end{aligned}
\end{equation}
where if $G(x,\xi_{\omega})\leq 0$,  $\mathbbm{1}_{(-\infty,0] }\left(G(x, \xi_{\omega})\right)=1$, otherwise $\mathbbm{1}_{(-\infty,0] }\left(G(x, \xi_{\omega})\right)=0$. 

\subsection{Additional Definitions and Assumptions}
\begin{definition}(Hausdorff Distances) \cite{ Liu19,Gibbs02}\label{def:set_distance}
Let $\|\cdot\|$ denote the 2-norm of a vector, and the Frobenius norm of a matrix. Let $A,B \subseteq \mathbb{R}^n$,

\begin{equation}
\begin{aligned}
 &\mathrm{d}(a,B):=\min_{b\in B} \|a-b\|, \quad D(A, B):=\max_{a\in A} \mathrm{d}(a,B),   
\end{aligned}   
\end{equation}
and 
\begin{equation}\label{beta_distance}
    \beta_{|\Omega|}:=D(\Xi,\Xi_{|\Omega|}).
\end{equation}

Let $\mathbb{P}, \mathbb{Q}\in \mathcal{D}$, and $\mathcal{G}$ be a family of Lipschitz continuous functions with constant 1, i.e.,
\begin{equation}
\begin{aligned}\label{define_functionG}
      \mathcal{G} = \{g\mid
|g(\xi')-g(\xi'')|\leq \|\xi'-\xi''\|,\; \forall \xi',\xi''\in \Xi\;\}. 
    \end{aligned}
\end{equation}

Then
\begin{equation}\label{define_distanceP}
    \rho(\mathbb{Q},\mathbb{P}):=\sup _{g \in \mathcal{G}}\left|\mathbb{E}_{\mathbb{Q}}[g(\xi)]-\mathbb{E}_{\mathbb{P}}[g(\xi)]\right|
\end{equation}
is the Kantorovich metric \cite{Edwards11} between probability distributions $\mathbb{P}, \mathbb{Q}$ (see Theorem~\ref{duality} for the relation between the Kantorovich metric and the Wasserstein metric).
The metric between the set ${\cal P}$ and $\mathbb{Q}$ is defined as:
\begin{equation}\label{define_distancePSet}
    \mathrm{d}(\mathbb{Q}, \mathcal{P}):=\min_{\mathbb{P} \in \mathcal{P}}\rho(\mathbb{Q},\mathbb{P}).
\end{equation}
Finally, the Hausdorff metric between $\mathcal{P}_{|\Omega|}$ and $\mathcal{P}$ is defined as 
\begin{equation}\label{define_distanceSet}
\mathrm{H}\left(\mathcal{P}_{|\Omega|}, \mathcal{P}\right):=\max \left\{\max_{\mathbb{P}_{|\Omega|}\in\mathcal{P}_{|\Omega|}} \mathrm{d}\left(\mathbb{P}_{|\Omega|}, \mathcal{P}\right), \max_{\mathbb{P}\in\mathcal{P}}  \mathrm{d}\left(\mathbb{P}, \mathcal{P}_{|\Omega|}\right)\right\} .
\end{equation}
\end{definition}

\subsubsection{Technical Assumption}
We make a general assumption on the discretization of the set $\mathcal{P}$. This assumption can be verified in special cases (Section~\ref{sec:examples}).
\begin{assumption}\label{ambiguity_assumption}
There exists a positive constant $C^H$, such that $\mathrm{H}\left(\mathcal{P}_{|\Omega|}, \mathcal{P}\right) \leq C^H\beta_{|\Omega|}$, where  $\mathrm{H}(\cdot,\cdot)$ is defined in \eqref{define_distanceSet} and $\beta_{|\Omega|}$ is defined in \eqref{beta_distance}. 
\end{assumption}

\subsection{Known Results}\label{sec:known_results}
We first introduce some results from
Liu et al. \cite{Liu19} in this section. These results provide the convergence rates of the bound $\beta_{|\Omega|}$ for different ways of generating  $\xi_1,\xi_2,\dots,\xi_\omega$ from $\Xi$. The first approach is to generate these points as samples from the uniform distribution over the support $\Xi$. 
\begin{proposition}\cite[Page 28]{Liu19}\cite[Section 2.2]{zhigljavsky2007stochastic}\label{beta_converge_dist}
Let $\xi_1, \xi_2,\dots,\xi_\omega$ be uniformly distributed on the support $\Xi\subseteq\mathbb{R}^d$. Then  $\beta_{|\Omega|}$
follows an extreme value distribution with
\begin{equation}\label{extreme_value_distribution}
\lim_{|\Omega|\rightarrow \infty}\frac{|\Omega|(2\beta_{|\Omega|})^d-\log {|\Omega|}}{\log\log {|\Omega|}}=d-1 \quad w.p.1,
\end{equation}
where $w.p.1$ means with probability 1.
\end{proposition}
Another approach to generating samples is to choose quantization points.
\begin{proposition}\cite[Proposition 9]{Liu19} \label{prop:beta_convergence_InProbability}
 Let $\xi_1, \dots, \xi_{|\Omega|}$ be independent and identically distributed samples generalized from a probability distribution $\xi$ on the support $\Xi\subseteq \mathbb{R}^d$. Assume that
(a) $\Xi$ is bounded; and
(b) the probability 
%not necessary be nomial distribution
distribution of $\xi$ is continuous, and there exist positive constants $V$, $v$ and $\delta_0$ such that
$
\text{Prob}\left(\left\|\xi-\xi_0\right\| \leq \delta\right)>V \delta^v
$
for any fixed point $\xi_0 \in \Xi$ and $\delta\in (0,\delta_0)$.
Then for any $\varepsilon>0$, for $|\Omega|$ sufficiently large, there exist positive constants $b(\varepsilon)$ and $B(\varepsilon)$ depending on $\varepsilon$ such that
\begin{equation}\label{quantizer_beta}
\operatorname{Prob}\left(\beta_{|\Omega|} \geq \varepsilon\right) \leq B(\varepsilon) exp^{-b(\varepsilon) |\Omega|}.
\end{equation} 
\end{proposition}
Proposition~\ref{beta_converge_dist} and \ref{prop:beta_convergence_InProbability} show convergence of the set $\Xi_{|\Omega|}$ to $\Xi$. Proposition~\ref{beta_converge_dist} explicitly provides a convergence rate, whereas Proposition~\ref{prop:beta_convergence_InProbability} indicates convergence in probability.

%\textcolor{blue}{Note: \cite{henrion2004holder} use an empirical measure to approximate the distribution of $\xi$. In our context, the samples are uniformly distributed or quantizers of the support $\Xi$. Although the quantizers may have a relationship with the empirical approximation measure, that relationship is unclear.}
\section{New Convergence 
 Results}\label{proof_DRO_convergence}
In this section, we state the main convergence analysis of the approximated model \eqref{dro_model_sampled} to \eqref{dro_model}.
We first show that the constraints and the optimal value of \eqref{dro_model_sampled} converge to that of \eqref{dro_model}(Proposition~\ref{convergence_chance_constraint} and Theorem~\ref{DRO_constraint}). Then we provide the convergence rate of \eqref{dro_model_sampled} to \eqref{dro_model} in Corollary~\ref{coro:beta_converge_dist} and \ref{coro:beta_convergence_InProbability} based on the two ways of generating samples from $\Xi$ as in Proposition~\ref{beta_converge_dist} and \ref{prop:beta_convergence_InProbability}.
We first give a bound on the distance between any distribution $\mathbb{P}$ in the ambiguity set $\mathcal{P}$ and its discretization in $\mathcal{P}_{|\Omega|}$.

\begin{proposition}\label{ambiguity_set_convergence}
For any $\mathbb{P}\in\mathcal{P}$, let $\mathbb{P}_{|\Omega|}$ be a solution of $\text{min}_{\mathbb{Q}\in\mathcal{P}_{|\Omega|}}\rho(\mathbb{P},\mathbb{Q})$. Then $\rho(\mathbb{P},\mathbb{P}_{|\Omega|})\leq C^H\beta_{|\Omega|}$.
\end{proposition}
\begin{proof}
    Assumption~\ref{ambiguity_assumption} implies that  $\forall\mathbb{P}\in\mathcal{P}$,
$\mathrm{d}\left(\mathbb{P}, \mathcal{P}_{|\Omega|}\right)\leq C^H\beta_{|\Omega|}$.
%Let $\mathbb{P}_{|\Omega|}=\text{argmin}_{\mathbb{Q}\in\mathcal{P}_{|\Omega|}}\rho(\mathbb{P},\mathbb{Q})$. 
Then from \eqref{define_distancePSet},
$\rho(\mathbb{P},\mathbb{P}_{|\Omega|})=\min_{\mathbb{Q}\in\mathcal{P}_{|\Omega|}}\rho(\mathbb{P},\mathbb{Q})=\mathrm{d}\left(\mathbb{P}, \mathcal{P}_{|\Omega|}\right)\leq C^H\beta_{|\Omega|}$.
$~\Box$
\end{proof}

Based on the result of Proposition~\ref{ambiguity_set_convergence}, we now show the convergence of the constraint satisfaction of \eqref{dro_model_sampled} to that of  \eqref{dro_model} for any distribution $\mathbb{P}\in\mathcal{P}$.

\begin{proposition}\label{convergence_chance_constraint} 
Let $\mathbb{P} \in \mathcal{P}$,  $\mathbb{P}_{|\Omega|}$ be a solution of $\text{min}_{\mathbb{Q}\in\mathcal{P}_{|\Omega|}}\rho(\mathbb{P},\mathbb{Q})$, and consider  \eqref{dro_model} and \eqref{dro_model_sampled}.
Then for a given $x$,
%for $|\Omega|$ sufficiently large, %there exists an $|\Omega|_0$, such that when $|\Omega|\geq |\Omega|_0$, 
 $$\left|\text{Prob}_\mathbb{P}(G(x,\xi)\leq 0)-\text{Prob}_{\mathbb{P}_{|\Omega|}}(G(x,\xi)\leq 0)\right|
 \leq \sqrt{2\kappa^GC^\mathcal{P}C^H\beta_{|\Omega|}}.$$
\end{proposition}
\begin{proof}
    For given constants $\kappa^L$ and $M^L$, let the set of all bounded and Lipshitz continuous functions be $\mathcal{L}=\{L\mid
||L(y')-L(y'')||\leq \kappa^L||y'-y''||,\; ||L(y')||\leq M^L,\; \;\forall y',y''\in \mathbb{R}^m,\;\; M^L, \kappa^L<\infty\}$. For any $L(\cdot)\in\mathcal{L}$ $x\in\mathcal{X}$, and $\xi',\xi''\in \Xi$:
\begin{equation}
\begin{aligned}
 &||L(G(x, \xi')) - L(G(x, \xi''))||
 \leq \kappa^L ||G(x, \xi')-G(x, \xi'')||
 \leq \kappa^L\kappa^G ||\xi'-\xi''||.
\end{aligned} 
\end{equation}
Hence, $L(G(x,\xi))$ is Lipschitz continuous on $\xi$ for the given $x$. 
Let $\mathcal{G}':=\{g(\cdot)=\frac{1}{\kappa^L\kappa^G}L(G(x,\cdot)),x\in\mathcal{X}\}\subseteq \mathcal{G}$. Then for any $L(\cdot)\in\mathcal{L}$ and $x\in\mathcal{X}$, according to Proposition~\ref{ambiguity_set_convergence},
\begin{equation}\label{lipschitz_bound}
\begin{aligned}
 &\quad \left|\mathbb{E}_{\mathbb{P}}\left[L(G(x, \xi))\right]-\mathbb{E}_{\mathbb{P}_{|\Omega|}}\left[L(G(x, \xi))\right]\right|\\
 & \leq \kappa^L\kappa^G\sup_{g \in \mathcal{G}'}\left|\mathbb{E}_{\mathbb{P}}[g(\xi)]-\mathbb{E}_{\mathbb{P}_{|\Omega|}}[g(\xi)]\right|\\
 &\leq \kappa^L\kappa^G\sup_{g \in \mathcal{G}}\left|\mathbb{E}_{\mathbb{P}}[g(\xi)]-\mathbb{E}_{\mathbb{P}_{|\Omega|}}[g(\xi)]\right|\\
 &=\kappa^L\kappa^G\rho(\mathbb{P},\mathbb{P}_{|\Omega|})
 \leq \kappa^L\kappa^GC^H\beta_{|\Omega|}.   
\end{aligned}
\end{equation}

The rest of the proof is similar to the proof of Proposition 1.2 in \cite{Ross11} (see also  \cite[Definition 4.1]{sun2019convergence}). We present it here for completeness. Let us consider $L_\epsilon(\cdot)\in\mathcal{L}$ to be:
\begin{equation}
    L_{\epsilon}(y)=\left\{\begin{array}{cc}
1 & y\leq 0 \\
1-y/\epsilon & 0<y\leq \epsilon\\
0 & y> \epsilon.
\end{array}\right.
\end{equation}
Note that for $L_{\epsilon}(\cdot)$, $\kappa^L=1/\epsilon$ and $M^L=1$. Since $L_{\epsilon}(G(x,\xi))\geq \mathbbm{1}_{(-\infty,0]}(G(x,\xi))$ for any $x\in\mathcal{X}$ and $\xi\in\Xi$,  we have  
\begin{equation}\label{epsilon_largerbound}
\mathbb{E}_{\mathbb{P}_{|\Omega|}}\left[L_\epsilon\left(G(x, \xi)\right)\right]\geq \mathbb{E}_{\mathbb{P}_{|\Omega|}}\left[\mathbbm{1}_{(-\infty,0]}\left(G(x, \xi)\right)\right].
\end{equation}
Also, because the probability density $f_\mathbb{P}(\cdot)$ of $\mathbb{P}$ is bounded by $C^\mathcal{P}$, we have 
\begin{equation}\label{epsilon_proabilityBound}
\begin{aligned}
 &\mathbb{E}_{\mathbb{P}}\left[L_\epsilon\left(G(x, \xi)\right)\right]-\mathbb{E}_{\mathbb{P}}\left[\mathbbm{1}_{(-\infty,0]}\left(G(x, \xi)\right)\right]
 =\int_0^\epsilon f_\mathbb{P}(y)\cdot\frac{y}{\epsilon}dy\leq \frac{C^\mathcal{P}}{\epsilon}\int_0^\epsilon ydy=\frac{C^\mathcal{P}\epsilon}{2}.   
\end{aligned}    
\end{equation}
Then we have the following inequalities:
\begin{equation}\label{epsilon_bound1}
    \begin{aligned}
&\mathbb{E}_{\mathbb{P}_{|\Omega|}}\left[\mathbbm{1}_{(-\infty,0]}\left(G(x, \xi)\right)\right]-\mathbb{E}_{\mathbb{P}}\left[\mathbbm{1}_{(-\infty,0]}\left(G(x, \xi)\right)\right]\\
=\;& \mathbb{E}_{\mathbb{P}_{|\Omega|}}\left[\mathbbm{1}_{(-\infty,0]}\left(G(x, \xi)\right)\right]-\mathbb{E}_{\mathbb{P}}\left[L_\epsilon\left(G(x, \xi)\right)\right]+\mathbb{E}_{\mathbb{P}}\left[L_\epsilon\left(G(x, \xi)\right)\right]-\mathbb{E}_{\mathbb{P}}\left[\mathbbm{1}_{(-\infty,0]}\left(G(x, \xi)\right)\right]\\
\leq\;& \mathbb{E}_{\mathbb{P}_{|\Omega|}}\left[L_\epsilon\left(G(x, \xi)\right)\right]-\mathbb{E}_{\mathbb{P}}\left[L_\epsilon\left(G(x, \xi)\right)\right]+\mathbb{E}_{\mathbb{P}}\left[L_{\epsilon}\left(G(x, \xi)\right)\right]-\mathbb{E}_{\mathbb{P}}\left[\mathbbm{1}_{(-\infty,0]}\left(G(x, \xi)\right)\right] 
\text{(Applying \eqref{epsilon_largerbound})}\\
\leq \;& \frac{\kappa^GC^H\beta_{|\Omega|}}{\epsilon}+\frac{C^\mathcal{P}\epsilon}{2}.
\;\qquad\qquad\qquad \text{(Applying \eqref{lipschitz_bound} and \eqref{epsilon_proabilityBound})}
    \end{aligned}\nonumber
\end{equation}
For $\epsilon=\sqrt{2\kappa^GC^H\beta_{|\Omega|}/{C^\mathcal{P}}}$, it gives 
%\eqref{epsilon_bound1} gives 
\begin{equation}
\begin{aligned}
& \mathbb{E}_{\mathbb{P}}\left[\mathbbm{1}_{(-\infty,0]}\left(G(x, \xi)\right)\right]-\mathbb{E}_{\mathbb{P}_{|\Omega|}}\left[\mathbbm{1}_{(-\infty,0]}\left(G(x, \xi)\right)\right]
 \leq \sqrt{2\kappa^GC^\mathcal{P}C^H\beta_{|\Omega|}}.   
\end{aligned}
\end{equation} 

As for the lower bound, we have the function $L_{-\epsilon}(\cdot)\in\mathcal{L}$: \begin{equation}
    L_{-\epsilon}(y)=\left\{\begin{array}{cc}
1 & y\leq -\epsilon \\
1-y/\epsilon & -\epsilon<y\leq 0\\
0 & y> 0.
\end{array}\right.
\end{equation}
For $L_{-\epsilon}(\cdot)$, $\kappa^L=1/\epsilon$ and $M^L=1$. 
Similarly, when we take $\epsilon=\sqrt{2\kappa^GC^H\beta_{|\Omega|}/{C^\mathcal{P}}}$, we have the lower bound:
\begin{equation}
\begin{aligned}
  & \mathbb{E}_{\mathbb{P}}\left[\mathbbm{1}_{(-\infty,0]}\left(G(x, \xi)\right)\right]-\mathbb{E}_{\mathbb{P}_{|\Omega|}}\left[\mathbbm{1}_{(-\infty,0]}\left(G(x, \xi)\right)\right]
   \geq -\sqrt{2\kappa^GC^\mathcal{P}C^H\beta_{|\Omega|}}. 
\end{aligned}
\end{equation}

Therefore, we can conclude that 
\begin{equation}
\begin{aligned}
  & \left|\mathbb{E}_{\mathbb{P}}\left[\mathbbm{1}_{(-\infty,0]}\left(G(x, \xi)\right)\right]-\mathbb{E}_{\mathbb{P}_{|\Omega|}}\left[\mathbbm{1}_{(-\infty,0]}\left(G(x, \xi)\right)\right]\right|
   \leq  \sqrt{2\kappa^GC^\mathcal{P}C^H\beta_{|\Omega|}}, 
\end{aligned}
\end{equation}
which is equivalent to 
\begin{equation}
\begin{aligned}
 &\left|\text{Prob}_\mathbb{P}(G(x,\xi)\leq 0)-\text{Prob}_{\mathbb{P}_{|\Omega|}}(G(x,\xi)\leq 0)\right|
 \leq  \sqrt{2\kappa^GC^\mathcal{P}C^H\beta_{|\Omega|}}.   
\end{aligned}  
\end{equation}
$~\Box$
\end{proof}
Note that the convergence result in Proposition~\ref{convergence_chance_constraint} is different from the stability results derived by \cite{henrion1999metric,henrion2004holder}. These papers discussed the stability of the set of solutions of the chance constraints to its perturbation. However, we focus on the convergence of the chance constraints with a distribution of $\xi$ and its discrete approximation for a given $x$.

We now show that the optimal value of \eqref{dro_model_sampled} converges to that of \eqref{dro_model}.
\begin{theorem}\label{DRO_constraint}
Assume that the solution sets of \eqref{dro_model} and \eqref{dro_model_sampled} are nonempty and compact with finite optimal values. For a given $\Xi_{|\Omega|}$ with sufficiently large $|\Omega|$, let $(\hat{x},\hat{\mathbb{P}}_{|\Omega|})$ be an optimal solution of \eqref{dro_model_sampled} and $\hat{v}_{|\Omega|}$ be the corresponding optimal value. Let $(x^*,\mathbb{P}^*)$ be an optimal solution of \eqref{dro_model} and $v^*$ be the corresponding optimal value. 
Then 
%$\hat{x}$ satisfies $\text{Prob}_\mathbb{P}(G(\hat{x},\xi)\leq 0)\geq 1-\theta - \sqrt{2\kappa^GC^\mathcal{P}C^H\beta_{|\Omega|}}$ and 
$|\hat{v}_{|\Omega|}-v^*|\leq \kappa^FC^H \beta_{|\Omega|}+\kappa^\theta\sqrt{2\kappa^GC^\mathcal{P}C^H\beta_{|\Omega|}}$.
\end{theorem}

\begin{proof}
    From Proposition~\ref{convergence_chance_constraint}, let $\mathbb{P}_{|\Omega|}^*$ be a solution of $\text{min}_{\mathbb{Q}\in\mathcal{P}_{|\Omega|}}\rho(\mathbb{P}^*,\mathbb{Q})$. If $x$ satisfies $ \text{Prob}_{\mathbb{P}_{|\Omega|}^*}(G(x,\xi)\leq 0)\geq 1-\theta$, then we have   
\begin{equation}
\begin{aligned}
1-\theta\leq & \; \text{Prob}_{\mathbb{P}_{|\Omega|}^*}(G(x,\xi)\leq 0)
  \leq  \; \text{Prob}_{\mathbb{P}^*}(G(x,\xi)\leq 0)+\sqrt{2\kappa^GC^\mathcal{P}C^H\beta_{|\Omega|}}.  
\end{aligned}
\end{equation}
Therefore,
\begin{equation}
\begin{aligned}\label{constraint_feasible_x}
  \text{Prob}_{\mathbb{P}^*}(G(x,\xi)\leq  0)\geq 1-\theta-\sqrt{2\kappa^GC^\mathcal{P}C^H\beta_{|\Omega|}}.
  \end{aligned}
\end{equation}

%We now discuss the convergence of the optimal values.
%For an optimal solution $x^*$, $\mathbb{P}^*$ and the corresponding optimal value $v^*$ of \eqref{dro_model},
Let $\epsilon=\sqrt{2\kappa^GC^\mathcal{P}C^H\beta_{|\Omega|}}$. Let $x'$ be the optimal solution of $v_{\mathbb{P}_{|\Omega|}^*}(\theta)$ and then $x'$ is feasible to $v_{\mathbb{P}^*}(\theta+\epsilon)$ by \eqref{constraint_feasible_x}.  Then by Assumption~\ref{basic_assumption}.4,
we have 
\begin{equation}\label{dro_converge_const}
\begin{aligned}
v^* =v_{\mathbb{P}^*}(\theta)\leq  v_{\mathbb{P}^*}(\theta+\epsilon)+\kappa^\theta\epsilon\leq \mathbb{E}_{\mathbb{P}^* }[F(x',\xi)] +\kappa^\theta\epsilon.
\end{aligned}
\end{equation}
For \eqref{dro_model_sampled}, we have 
\begin{equation}\label{dro_converge_const2}
\begin{aligned}
%\hat{v}_{|\Omega|} = \max_{\mathbb{P}_{|\Omega|}\in \mathcal{P}_{|\Omega|}}\mathbb{E}_{\mathbb{P}_{|\Omega|} }[F(\hat{x},\xi)]
\hat{v}_{|\Omega|} = v_{\hat{\mathbb{P}}_{|\Omega|}}(\theta)\geq v_{\mathbb{P}_{|\Omega|}^*}(\theta)=\mathbb{E}_{\mathbb{P}_{|\Omega|}^* }[F(x',\xi)].
\end{aligned}
\end{equation}
The following proof is similar to that of \cite[Theorem 14]{Liu19}. Let $\mathcal{G}'':=\{g(\cdot)=\frac{1}{\kappa^F}F(x,\cdot),x\in\mathcal{X}\}\subseteq \mathcal{G}$. Since $\mathcal{P}_{|\Omega|}\subseteq\mathcal{P}$, we have:
\begin{equation}
\begin{aligned}
0\leq &\; v^*-\hat{v}_{|\Omega|}\\
%= \; \textcolor{blue}{ v_{\mathbb{P}^*}(\theta)- v_{\hat{\mathbb{P}}_{|\Omega|}}(\theta)}\\
%\leq &\; \textcolor{blue}{v_{\mathbb{P}^*}(\theta+\epsilon)-v_{\mathbb{P}_{|\Omega|}^*}(\theta)+\kappa^\theta\epsilon} \\
\leq & \; \mathbb{E}_{\mathbb{P}^* }[F(x',\xi)]- \mathbb{E}_{\mathbb{P}_{|\Omega|}^* }[F(x',\xi)]+\kappa^\theta\epsilon \qquad \text{(Applying \eqref{dro_converge_const} and \eqref{dro_converge_const2})}\\
 \leq &\; \max_{\mathbb{P}\in \mathcal{P}}\min_{\mathbb{P}_{|\Omega|}\in \mathcal{P}_{|\Omega|}}\; \left(\mathbb{E}_{\mathbb{P} }[F(x',\xi)]- \mathbb{E}_{\mathbb{P}_{|\Omega|} }[F(x',\xi)]\right)+\kappa^\theta\epsilon\\ %&&\text{(Reorganizing the formula)}\\
% \leq & \;\max_{\mathbb{P}\in \mathcal{P}}\min_{\mathbb{P}_{|\Omega|}\in \mathcal{P}_{|\Omega|}}\max_{x\in\mathcal{X}}\; \left|\mathbb{E}_{\mathbb{P}_{|\Omega|} }[F(x,\xi)]- \mathbb{E}_{\mathbb{P} }[F(x,\xi)]\right|\\%&&\text{(Taking the maximum on $x$)}\\
     \leq  &\; \kappa^F\max_{\mathbb{P}\in \mathcal{P}}\min_{\mathbb{P}_{|\Omega|}\in \mathcal{P}_{|\Omega|}}\sup_{g\in\mathcal{G}''}\; \left|\mathbb{E}_{\mathbb{P}_{|\Omega|} }[g(\xi)]- \mathbb{E}_{\mathbb{P} }[g(\xi)]\right|+\kappa^\theta\epsilon\\
     \leq  &\;\kappa^F\max_{\mathbb{P}\in \mathcal{P}}\min_{\mathbb{P}_{|\Omega|}\in \mathcal{P}_{|\Omega|}}\sup_{g\in\mathcal{G}}\; \left|\mathbb{E}_{\mathbb{P}_{|\Omega|} }[g(\xi)]- \mathbb{E}_{\mathbb{P} }[g(\xi)]\right|+\kappa^\theta\epsilon\\%&&\text{(Replacing $F(x,\cdot)$ by $g(\cdot)$)}\\
     %\leq &\; \kappa^F \max_{\mathbb{P}\in \mathcal{P}}\min_{\mathbb{P}_{|\Omega|}\in \mathcal{P}_{|\Omega|}}\;  \rho(\mathbb{P}_{|\Omega|}, \mathbb{P})= \kappa^F\max_{\mathbb{P}\in \mathcal{P}} \mathrm{d}(\mathbb{P}, \mathcal{P}_{|\Omega|})\\
    % \text{(Applying \eqref{define_distanceP} and \eqref{define_distancePSet})}\\
    \leq & \kappa^F \mathrm{H}\left(\mathcal{P}_{|\Omega|}, \mathcal{P}\right)+\kappa^\theta\epsilon\leq \kappa^FC^H\beta_{|\Omega|}+\kappa^\theta\epsilon. \text{(Applying Assumption~\ref{ambiguity_assumption})}
    \end{aligned}\nonumber
\end{equation}
%For the other direction, we apply \eqref{dro_converge_const2} in the inequalities and have the lower bound:
%\begin{equation}
%$  v^*-\hat{v}_{|\Omega|} \geq  -\kappa^FC^H\beta_{|\Omega|}. $%\end{equation}
$~\Box$
\end{proof}

The following corollary presents the convergence rate of \eqref{dro_model_sampled} to \eqref{dro_model} when the samples are generated to follow the uniform distribution on $\Xi$. 
\begin{corollary}\label{coro:beta_converge_dist}
Let $\xi_1, \xi_2,\dots,\xi_\omega$ be uniformly distributed on the support $\Xi\subseteq\mathbb{R}^d$ as in Proposition~\ref{beta_converge_dist}. For any given $\varepsilon>0$, there exists a $|\Omega|_0^\varepsilon$ sufficient large. Then when $|\Omega|>|\Omega|_0^\varepsilon$,  

$|v^*-\hat{v}_{|\Omega|}|<\frac{\kappa^FC^H}{2}\left(\frac{\log |\Omega|+(d-1+\varepsilon)\log \log |\Omega|}{|\Omega|}\right)^{1/d}+\kappa^\theta\sqrt{\kappa^GC^\mathcal{P}C^H\left(\frac{\log |\Omega|+(d-1+\varepsilon)\log \log |\Omega|}{|\Omega|}\right)^{1/d}}$ w.p.1.
%and \quad 
%$\text{Prob}_\mathbb{P}(G(\hat{x},\xi)\leq  0)
%>1-\theta-\sqrt{\kappa^GC^\mathcal{P}C^H\left(\frac{\log |\Omega|+(d-1+\varepsilon)\log \log |\Omega|}{|\Omega|}\right)^{1/d}}.       $ 
%    \end{aligned}\nonumber
%\end{equation}

\end{corollary}

\begin{proof}
    We rewrite \eqref{extreme_value_distribution} as follows. For any $\varepsilon>0$,  there exists $|\Omega|$ sufficiently large, i.e. 
there exists a $|\Omega|_0^\varepsilon$, such that for all $|\Omega|>|\Omega|_0^\varepsilon$, 
$$\left|\frac{|\Omega|(2\beta_{|\Omega|})^d-\log {|\Omega|}}{\log\log {|\Omega|}}-d+1\right|< \varepsilon.$$
By reorganizing the inequality, we have 
$$\beta_{|\Omega|}< \frac{1}{2}\left(\frac{\log |\Omega|+(d-1+\varepsilon)\log \log |\Omega|}{|\Omega|}\right)^{1/d}\quad w.p.1.$$
Since  $|v^*-\hat{v}_{|\Omega|}|\leq  \kappa^FC^H\beta_{|\Omega|}+\kappa^\theta\sqrt{2\kappa^GC^\mathcal{P}C^H\beta_{|\Omega|}}$ from Theorem~\ref{DRO_constraint}, w.p.1,
\begin{equation}
    \begin{aligned}
 &|v^*-\hat{v}_{|\Omega|}|<\frac{\kappa^FC^H}{2}\left(\frac{\log |\Omega|+(d-1+\varepsilon)\log \log |\Omega|}{|\Omega|}\right)^{1/d}\\
 &+\kappa^\theta\sqrt{\kappa^GC^\mathcal{P}C^H\left(\frac{\log |\Omega|+(d-1+\varepsilon)\log \log |\Omega|}{|\Omega|}\right)^{1/d}}.       \nonumber
    \end{aligned}
\end{equation}
%According to \eqref{constraint_feasible}, we have 
%\begin{equation}
%    \begin{aligned}
%&\text{Prob}_\mathbb{P}(G(\hat{x},\xi)\leq  0)\\
%>&1-\theta-\sqrt{\kappa^GC^\mathcal{P}C^H\left(\frac{\log |\Omega|+(d-1+\varepsilon)\log \log |\Omega|}{|\Omega|}\right)^{1/d}}.        
%    \end{aligned}\nonumber
%\end{equation}
$~\Box$
\end{proof}

The following corollary presents the convergence rate of \eqref{dro_model_sampled} to \eqref{dro_model} when the samples are generated by choosing quantization points. 
\begin{corollary}\label{coro:beta_convergence_InProbability}
 Let $\xi_1, \dots, \xi_{|\Omega|}$ be independent and identically distributed samples generalized from $\xi$ as in Proposition~\ref{prop:beta_convergence_InProbability}.
Then for any $\varepsilon>0$, when $|\Omega|$ is sufficiently large, there exist positive constants $b(\varepsilon)$ and $B(\varepsilon)$ depending on $\varepsilon$ such that
$
\text{Prob}\left(|v^*-\hat{v}_{|\Omega|}|\geq \kappa^FC^H\varepsilon+\kappa^\theta\sqrt{2\kappa^GC^\mathcal{P}C^H\varepsilon}\right) \leq B(\varepsilon) exp^{-b(\varepsilon) |\Omega|}
$.
%and \quad $\operatorname{Prob}\Big(\text{Prob}_\mathbb{P}(G(\hat{x},\xi)\leq  0) \leq 1-\theta- \sqrt{2\kappa^GC^\mathcal{P}C^H\varepsilon}\Big) 
%     \leq \; B(\varepsilon) exp^{-b(\varepsilon) |\Omega|}.  $
\end{corollary}
\begin{proof}
    From Theorem~\ref{DRO_constraint}, we know 
$|v^*-\hat{v}_{|\Omega|}|\leq  \kappa^FC^H\beta_{|\Omega|}+\kappa^\theta\sqrt{2\kappa^GC^\mathcal{P}C^H\beta_{|\Omega|}}$.
Clearly, for any $\varepsilon>0$ and sufficiently large  $|\Omega|$, from \eqref{quantizer_beta} in Proposition~\ref{prop:beta_convergence_InProbability},  we have
\begin{equation}\label{convergence_rate_prop2}
\begin{aligned}
&\text{Prob}\left(|v^*-\hat{v}_{|\Omega|}|\geq \kappa^FC^H\varepsilon+\kappa^\theta\sqrt{2\kappa^GC^\mathcal{P}C^H\varepsilon}\right)\\
\leq &\;  \text{Prob}\Big(\kappa^FC^H\beta_{|\Omega|}+\kappa^\theta\sqrt{2\kappa^GC^\mathcal{P}C^H\beta_{|\Omega|}}\geq \kappa^FC^H\varepsilon+\kappa^\theta\sqrt{2\kappa^GC^\mathcal{P}C^H\varepsilon}\Big)\\
%\leq &\;  \text{Prob}\big(\beta_{|\Omega|}\geq \varepsilon\big)+\text{Prob}\big(\sqrt{\beta_{|\Omega|}}\geq \sqrt{\varepsilon}\big)\\
\leq &\;B(\varepsilon) exp^{-b(\varepsilon) |\Omega|}.
\end{aligned}
\end{equation}
The second inequality holds because $\kappa^FC^H\beta_{|\Omega|}+\kappa^\theta\sqrt{2\kappa^GC^\mathcal{P}C^H\beta_{|\Omega|}}$ is monotonically increasing in $\beta_{|\Omega|}$.  
%\begin{equation}\label{convergence_rate_prop2}
%\text{Prob}\left(\frac{|v^*-\hat{v}_{|\Omega|}|}{\kappa^FC^H}\geq \varepsilon\right)\leq  \text{Prob}\left(\beta_{|\Omega|}\geq \varepsilon\right)\leq B(\varepsilon) exp^{-b(\varepsilon) |\Omega|}. 
%\end{equation}
%Then according to \eqref{constraint_feasible}, for sufficiently large $|\Omega|$, we have 
%\begin{equation}\nonumber
%    \begin{aligned}
     %&\operatorname{Prob}\left(\frac{1-\theta- \text{Prob}_\mathbb{P}(G(\hat{x},\xi)\leq  0)}{\sqrt{2\kappa^GC^\mathcal{P}C^H}} \geq \sqrt{\varepsilon}\right)\\
%     \leq & \;\text{Prob}\left(\sqrt{\beta_{|\Omega|}}\geq \sqrt{\varepsilon}\right)\leq B(\varepsilon) exp^{-b(\varepsilon) |\Omega|}.   
%    \end{aligned}
%\end{equation}
$~\Box$
\end{proof}
The bounds in this section are worst-case bounds and such bounds allow us to develop a formal proof of convergence. The actual number of sample points can be determined empirically for specific problems. Empirically, we have observed that for $\Xi\subseteq \mathbb{R}^{10}$, this number is about 500 to 1000.

\section{Upper and Lower Bounds}
Shapiro and Philpott \cite{Shapiro07} provide a statistical procedure for estimating the optimality gap in the SAA of a stochastic program. The situation in the distributional robust context is different since the samples in solutions of the inner problems with the discrete approximation may not be equally weighted for each scenario.
To calculate a statistical estimate for the upper bound for the objective function of \eqref{dro_model_sampled} with $100(1-\alpha) \%$ confidence, we modify the statistical procedure in \cite{Shapiro07} by generating multiple batches of size $|\Omega|$, instead of using samples from a single batch. Also, compared to the method in \cite{Xu18} where the authors assumed the ambiguity set to be moment-based and calculated the upper bound by Lagrange duality, the following method allows us to estimate the upper bounds for any ambiguity sets. 
For the feasible solution set $\mathcal{X}'$ and a given feasible solution $\bar{x}\in\mathcal{X}'$, we know that 
$$q_{|\Omega|}(\bar{x})\geq \min_{x\in\mathcal{X}'} q_{|\Omega|}(x)=\hat{v}.$$
Using independently generated batches $|\Omega|^{(1)},\dots,|\Omega|^{(M)}$, we solve the subproblem $q_{|\Omega|}(\bar{x})$ with given feasible solution $\tilde{x}$ to optimality $M$ times. Let $\hat{q}^{(1)}_{|\Omega|}(\bar{x}),\dots,\hat{q}^{(M)}_{|\Omega|}(\bar{x})$ be the computed optimal values of the subproblems. 
Then an upper bound for  $ \mathbb{E}[q_{|\Omega|}(\hat{x})]$ with $100(1-\alpha)\%$ confidence is given as: 
\begin{equation}
    U_{|\Omega|,M}(\bar{x}):=\frac{1}{M}\sum_{m=1}^M \hat{q}^{(m)}_{|\Omega|}(\bar{x})+ t_{\alpha,\nu}\hat{\sigma}_{|\Omega|,M}(\bar{x}),\nonumber
\end{equation}
where the variance  $\hat{\sigma}^2_{|\Omega|,M}$ is estimated as 
\begin{equation}
\hat{\sigma}_{|\Omega|,M}^2(\bar{x}):=\frac{1}{M\left(M-1\right)} \sum_{m=1}^{M}\left[\hat{q}^{(m)}_{|\Omega|}(\bar{x})-\frac{1}{M}\sum_{m=1}^M \hat{q}^{(m)}_{|\Omega|}(\bar{x})\right]^2.\nonumber
\end{equation}

The lower bound estimation procedure is similar to that in \cite{Shapiro07}. The optimal value of \eqref{dro_model_sampled} based on an approximation using samples of size $|\Omega|$ is $\hat{v}_{|{\Omega}|}$.
For any $x\in\mathcal{X}'$, we have 
$$q_{|\Omega|}(x)= \mathbb{E}[\hat{q}_{|\Omega|}(x)]\geq \mathbb{E}[\min_{x'\in\mathcal{X}'}\hat{q}_{|\Omega|}(x')]= \mathbb{E}[\hat{v}_{|\Omega|}].$$
We solve the approximated model \eqref{dro_model_sampled} based on independently generated replicates
%,each of size $N$, 
$M'$ times and $\hat{v}_{|\Omega|}^{(1)},...,\hat{v}_{|\Omega|}^{(M')}$ are the computed optimal values of these models. 
The $100(1-\alpha)\%$ lower bound for $\mathbb{E}[\hat{v}_{|\Omega|}]$ is given by
\begin{equation}
L_{|\Omega|, M'}:=\frac{1}{M'} \sum_{m=1}^{M'} \hat{v}_{|\Omega|}^{(m)}-t_{ \alpha, \nu} \hat{\sigma}_{|\Omega|, M'},\nonumber
\end{equation}
where the variance estimate is defined as
\begin{equation}
\hat{\sigma}_{|\Omega|, M}^2:=\frac{1}{M'(M'-1)} \sum_{m=1}^{M'}\left(\hat{v}_{|\Omega|}^{(m)}-\frac{1}{M'} \sum_{m=1}^{M'} \hat{v}_{|\Omega|}^{(m)}\right)^2.\nonumber
\end{equation}

We can use small values of $M$ and $M'$, for example, $M=M'=10$ in applications, and use $t_{\alpha, \nu}$ as the $\alpha$-critical value of the $t$-distribution with $\nu$ degrees of freedom, where $\nu=M-1$.

\section{Examples}\label{sec:examples}
Assumption~\ref{ambiguity_assumption} is a key assumption in the convergence analysis developed in Section~\ref{proof_DRO_convergence}. We now give several examples and discuss how this assumption can be satisfied. There are various ways to construct the ambiguity set \cite{Rahimian19} and we review some of them to present a convergence analysis of $\mathcal{P}_{|\Omega|}$ to $\mathcal{P}$.

\subsection{Moment-Based Ambiguity Set}

Assume that the ambiguity set is defined as:
\begin{equation}\label{ambigutyset_moment}
    \begin{aligned}
\mathcal{P} &=\{\mathbb{P}\mid\mathbb{E}_{\mathbb{P}}[\psi(\xi)]\in\mathcal{K} \},\\
\mathcal{P}_{|\Omega|}&=\{\mathbb{P}_{|\Omega|}\mid  \mathbb{P}_{|\Omega|}=\left(p_1,\dots,p_{|\Omega|}\right),\sum_{\omega=1}^{|\Omega|}p_\omega=1,  p_\omega\geq 0\; \forall \omega, \mathbb{E}_{\mathbb{P}_{|\Omega|}}[\psi(\xi)]= \sum_{\omega=1}^{|\Omega|}p_\omega\psi(\xi_\omega)\in\mathcal{K}\}.
    \end{aligned}
\end{equation}
where $\psi(\cdot)$ is a mapping consisting of vectors and/or matrices and $\mathcal{K}$ is a closed convex cone in a matrix space.  \cite{{Liu19}} provides specific conditions to ensure that Assumption~\ref{ambiguity_assumption} is satisfied for the moment-based ambiguity set \eqref{ambigutyset_moment}. 
Note that the ambiguity set \eqref{ambigutyset_moment} is similar to the ambiguity set in Chen, Sun and Xu \cite[Theorem 1]{chen2021decomposition}, where a more general set with additional equality constraints is considered.
%\begin{equation}\label{ambiguity_equalmoment}
%\begin{aligned}
%\mathcal{P} &=\{\mathbb{P}\mid\mathbb{E}_{\mathbb{P}}[\psi_E(\xi)]=\mu_E, \mathbb{E}_{\mathbb{P}}[\psi_I(\xi)]\preceq\mu_I \},   
%\end{aligned}
%\end{equation}
%where $\mu_E$ represents the mean of the random vector/matrix $\psi_E(\xi)$ and $\mu_I$ represents the upper bound of the mean of the random vector/matrix $\psi_I(\xi)$. The ambiguity set \eqref{ambiguity_equalmoment} considers the moment systems with equality constraints. However, we focus on the bounds of the moment systems, which only contain inequality constraints.

\begin{assumption}
\cite[Assumption 1]{Liu19}\label{def_slater}
For a given mapping $\psi(\cdot)$, there exists $\mathbb{P}_0 \in\mathcal{P}$ and a constant $\alpha>0$ such that
$\mathbb{E}_{\mathbb{P}_0}[\psi(\xi)] +\alpha\mathcal{B} \subset \mathcal{K},$ 
where $\mathcal{B}$ is the unit ball in the space of $\mathcal{K}$ and $+$ is the Minkowski sum.
\end{assumption}

\begin{theorem}\cite[Theorem 12]{Liu19}\label{moment_set_converge}
Let $\mathcal{P}$ and $\mathcal{P}_{|\Omega|}$ be defined as in \eqref{ambigutyset_moment}. If  $\mathcal{P}$ satisfies Assumption~\ref{def_slater},   
%Let the diameter of $\Xi$ be $\delta_\Xi$.
%$\beta_{|\Omega|}\rightarrow 0$  as ${|\Omega|} \rightarrow \infty$.
and the function $\psi(\xi)$ is Lipschitz continuous with Lipschitz constant $\kappa^\psi$,
then for sufficiently large ${|\Omega|}$, 
$\mathrm{H}\left(\mathcal{P}_{|\Omega|}, \mathcal{P}\right) \leq \left(1+\frac{2\kappa^\psi||\textbf{1}||M_\Xi}{\alpha}\right)\beta_{|\Omega|}$, where $\bf{1}$ is a matrix that has the same size as $\psi(\cdot)$ with each component being 1.
\end{theorem}

\subsection{Mean and Variance Ambiguity Set}

Let $\Sigma_0\succ0$ and consider the mean and variance ambiguity sets:
\begin{equation}\label{ambigutyset_variance}
    \begin{aligned}
     \mathcal{P}& =\{\mathbb{P}\mid \mu_0-\gamma_R I\leq \mathbb{E}_{\mathbb{P}}[\xi]\leq \mu_0+\gamma_L I, \mathbb{E}_{\mathbb{P}}[\xi\xi^T]-\mathbb{E}_{\mathbb{P}}[\xi]\mathbb{E}_{\mathbb{P}}[\xi]^T\preceq \gamma_s\Sigma_0\},\\
\mathcal{P}_{|\Omega|} & =\{\mathbb{P}_{|\Omega|}\mid\mathbb{P}_{|\Omega|}=\left(p_1,\dots,p_{|\Omega|}\right),\sum_{\omega=1}^{|\Omega|}p_\omega=1, p_\omega\geq 0\; \forall \omega, \mu_0-\gamma_R I\leq \sum_{\omega=1}^{|\Omega|}p_\omega\xi_\omega\leq \mu_0+\gamma_L I,\\
&\sum_{\omega=1}^{|\Omega|}p_\omega\xi_\omega\xi^T_\omega-\left(\sum_{\omega=1}^{|\Omega|}p_\omega\xi_\omega\right)\left(\sum_{\omega=1}^{|\Omega|}p_\omega\xi_\omega\right)^T\preceq\gamma_s\Sigma_0\},   
    \end{aligned}
\end{equation}
where $\gamma_R$, $\gamma_L$ and $\gamma_s$ are positive numbers and $\gamma_s>1$, $\mu_0$ is the given sample mean vector and $\Sigma_0$ is the given sample variance matrix. 

Note that the ambiguity set \eqref{ambigutyset_variance} is different from the ambiguity set in Delage and Ye \cite{Delage10}, where the ambiguity set is defined as
\begin{equation}\label{ambiguity_2moment}
\begin{aligned}
 \mathcal{P}&:=\{\mathbb{P}\mid
\mathbb{E}_{\mathbb{P}}\left[\xi-\mu_0\right]^T \Sigma_0^{-1} \mathbb{E}_{\mathbb{P}}\left[\xi-\mu_0\right] \leq \gamma_1,\mathbb{E}_{\mathbb{P}}\left[\left(\xi-\mu_0\right)\left(\xi-\mu_0\right)^T\right] \preceq \gamma_2 \Sigma_0
\},   
\end{aligned}
\end{equation}
with nonnegative constants $\gamma_1$ and $\gamma_2$, and $\mu_0$ and $\Sigma_0$ are the sample mean and sample covariance.
The ambiguity set \eqref{ambiguity_2moment} forces the mean and the centered second-moment matrix of $\xi$ to be within the upper bounds. However, in \eqref{ambigutyset_variance}, we allow the mean to be unspecified and bound the overall variance. The benefit from the ambiguity set of constraining variance is that in practice, variance directly measures the dispersion of a distribution from its mean, which in our case is unknown.

\begin{theorem}
Let $\mathcal{P}$ and $\mathcal{P}_{|\Omega|}$ be defined as in \eqref{ambigutyset_variance}.
For sufficiently large ${|\Omega|}$, 
$\mathrm{H}\left(\mathcal{P}_{|\Omega|}, \mathcal{P}\right) \leq \left(1+\frac{2\sqrt{2+16M_\Xi^2}M_\Xi\sqrt{d^2+2d}}{\alpha}\right)\beta_{|\Omega|}$.
\end{theorem}
\begin{proof}
We can formulate \eqref{ambigutyset_variance} as the same structure as \eqref{ambigutyset_moment}, where $\psi(\xi)$ is formulated as:
\begin{equation}
    %\psi(\xi)=\left(\begin{array}{c}
\psi(\xi):=\left(\begin{array}{cc}
\xi- \mu_0-\gamma_L\\
\mu_0-\xi- \gamma_R\\
(\xi-\mu_0)(\xi-\mu_0)^T-(\mathbb{E}[\xi]-\mu_0)(\mathbb{E}[\xi]-\mu_0)^T-\gamma_s\Sigma_0 \\
%\left(\xi-\mu_0\right)\left(\xi-\mu_0\right)^T-\gamma_2 \Sigma_0
%\end{array}\right)
\end{array}\right)\nonumber
\end{equation}
Then the ambiguity set $\mathcal{P}$ in \eqref{ambigutyset_variance} satisfies Assumption~\ref{def_slater} with $\alpha=\min\{\gamma_L,\gamma_R,(\gamma_s-1)\lambda_{\min}\}$, where $\lambda_{\min}$ is the smallest eigenvalue of $\Sigma_0$. 
Note that for $\xi',\xi''\in\Xi$, 
\begin{equation}\label{matrixproductTerm}
\begin{aligned}
&\quad \| \xi'(\xi')^T-\xi''(\xi'')^T\|  \\
& =  \| \xi'(\xi')^T-\xi'(\xi'')^T+\xi'(\xi'')^T- \xi''(\xi'')^T\|\\
%& =  \| \xi'((\xi')^T-(\xi'')^T)+(\xi'-\xi'')(\xi'')^T\|\\
& \leq \| \xi'((\xi')^T-(\xi'')^T)\|+\|(\xi'-\xi'')(\xi'')^T|\|\\
%\text{(Triangle Inequality)}\\
&\leq \| \xi'\|\cdot\|(\xi')^T-(\xi'')^T\|+\|\xi'-\xi''\|\cdot\|(\xi'')^T\|\\
%& \qquad\qquad\qquad\qquad\qquad\text{($\|AB\|\leq \|A\|\cdot\|B\|$)}\\
%&= \| \xi'\|\cdot\|\xi'-\xi''\|+\|\xi'-\xi''\|\cdot\|\xi''\|%&&\text{(Applying \eqref{norm1_inequaility})}\\
& \leq 2M_\Xi\|\xi'-\xi''\|. \qquad\text{($\|\xi\|\leq M_\Xi,\; \forall \xi\in\Xi$)}
\end{aligned}   
\end{equation}

We now have the following:
\begin{equation}
\begin{aligned}
&\quad \|\psi(\xi')-\psi(\xi'')\|\\
&=\left|\left|\left(\begin{array}{cc}
\xi'- \xi''\\
\xi'- \xi''\\
(\xi'-\mu_0)(\xi'-\mu_0)^T-(\xi''-\mu_0)(\xi''-\mu_0)^T
 \\
%\left(\xi-\mu_0\right)\left(\xi-\mu_0\right)^T-\gamma_2 \Sigma_0
%\end{array}\right)
\end{array}\right)\right|\right|\\
&=\left(2\|\xi'- \xi''\|^2 +\| \xi'(\xi')^T-\xi''(\xi'')^T-2\mu_0(\xi'-\xi'')^T\|^2\right)^{1/2}\\
%&\leq\big(2\|\xi'- \xi''\|^2+(\| \xi'(\xi')^T-\xi''(\xi'')^T\|+2\|\mu_0(\xi'-\xi'')^T\|)^2\big)^{1/2}\\
%\qquad\qquad\quad \text{(Triangle Inequality)}\\
&\leq \big(2\|\xi'-\xi''\|^2+(\| \xi'(\xi')^T-\xi''(\xi'')^T\|+2M_\Xi\|(\xi'-\xi'')^T\|)^2\big)^{1/2}   \quad\;\text{($\|\mu_0\|\leq M_\Xi$)}\\
%&\leq 2||\xi'-\xi''||+|| \xi'(\xi')^T-\xi''(\xi'')^T||+2M_\Xi||\xi'-\xi''|| \\
&\leq\big( 2\|\xi'-\xi''\|^2+16M_\Xi^2\|\xi'-\xi''\|^2 \big)^{1/2}\qquad \qquad\qquad\qquad\qquad\qquad\; \text{(Using \eqref{matrixproductTerm})}\\
%&\leq 2||\xi'- \xi''||+\max\{||\xi'- \xi''||, ||\xi'- \xi''||+2M_\Xi||\xi'-\xi''||\}\qquad \text{(Applying \eqref{norm1_inequaility})}\\
&\leq \sqrt{2+16M_\Xi^2}\|\xi'-\xi''\|.
\end{aligned}\nonumber
\end{equation}
Therefore, $\psi(\cdot)$ is Lipschitz continuous with Lipschitz constant $\kappa^\psi=\sqrt{2+16M_\Xi^2}$. 
%Let  $\mathcal{I}$ be a matrix of size $\psi(\cdot)$ with each component being 1. 
Since $\xi\in \Xi\subseteq \mathbb{R}^d$, $\|\textbf{1}\|= \sqrt{d+d+d^2}=\sqrt{d^2+2d}$. From Theorem~\ref{moment_set_converge},  when ${|\Omega|}$ is sufficiently large, 
\begin{equation}\nonumber
    \begin{aligned}
    & \mathrm{H}\left(\mathcal{P}_{|\Omega|}, \mathcal{P}\right) \leq \left(1+\frac{2\kappa^\psi\|\textbf{1}\|M_\Xi}{\alpha}\right)\beta_{|\Omega|}
     =   \left(1+\frac{2\sqrt{2+16M_\Xi^2}M_\Xi\sqrt{d^2+2d}}{\alpha}\right)\beta_{|\Omega|}.   
    \end{aligned}
\end{equation}
$~\Box$
\end{proof}

\subsection{$l_n$-Wasserstein Ambiguity Set}
Chen, Sun and Xu \cite[Theorem 2]{chen2021decomposition} provided a convergence result for the $l_1$-Wasserstein ambiguity sets. For the ambiguity sets
\begin{equation}\label{ambigutyset_distance_general}
\begin{aligned}
 \mathcal{P} &=\{\mathbb{P}\mid \mathbb{W}_1(\mathbb{P},\mathbb{P}') \leq d_1\}, \quad \text{and}\\
 \mathcal{P}_{|\Omega|} &=\{\mathbb{P}_{|\Omega|}\mid \mathbb{W}_1(\mathbb{P}_{|\Omega|},\mathbb{P}'_{|\Omega|})\leq d_2, \mathbb{P}_{|\Omega|}=\left(p_1,\dots,p_{|\Omega|}\right),\sum_{\omega=1}^{|\Omega|}p_\omega=1, p_\omega\geq 0\; \forall \omega  \}, 
\end{aligned}\nonumber
 \end{equation}
$\mathrm{H}\left(\mathcal{P}_{|\Omega|}, \mathcal{P}\right)\leq \rho(\mathcal{P}_{|\Omega|}, \mathcal{P})+2\beta_{|\Omega|}+|d_1-d_2|$.
Here we consider the special case when $d_0=d_1=d_2$ and $\mathbb{P}'=\mathbb{P}'_{|\Omega|}$ and extend the result to a more general case for the $l_n$-Wasserstein ambiguity sets.
Let
\begin{equation}\label{ambigutyset_distance}
\begin{aligned}
 \mathcal{P} &=\{\mathbb{P}\mid \mathbb{W}_n(\mathbb{P},\mathbb{P}') \leq d_0\}, \quad \text{and}\\
 \mathcal{P}_{|\Omega|} &=\{\mathbb{P}_{|\Omega|}\mid \mathbb{W}_n(\mathbb{P}_{|\Omega|},\mathbb{P}')\leq d_0, \mathbb{P}_{|\Omega|}=\left(p_1,\dots,p_{|\Omega|}\right),\sum_{\omega=1}^{|\Omega|}p_\omega=1, p_\omega\geq 0\; \forall \omega  \}, 
\end{aligned}
 \end{equation}
 where $\mathbb{P}'$ is the nominal probability distribution of $\xi$, $d_0 \geq 0$ is the Wasserstein radius and $\mathbb{W}_n$ is the $l_n$-Wasserstein distance \cite[Definition 2.5]{Pflug14}.  
 The $l_n$-Wasserstein metric between $\mathbb{P},\mathbb{Q}\in \mathcal{D}$ is defined as  
 \begin{equation}\label{def_ln_wasserstein}
\mathbb{W}_n(\mathbb{Q},\mathbb{P})=\left(\inf_\pi \iint_{\Xi\times\Xi}\left\|\xi-\xi'\right\|^n \pi\left(d \xi, d \xi'\right)\right)^{1/n},  \; n\geq 1,\nonumber
\end{equation}
where $\pi$ is a joint distribution of $\xi$ and $\xi'$ with marginal $\mathbb{P}$ and $\mathbb{Q}$.

We first introduce several properties of the $l_n-$Wasserstein metric. 
 \begin{lemma}(Monotonicity and Convexity of the $l_n-$Wasserstein metric)\label{properties_ln_wasserstein}

(i)\cite[Section 2.3]{Panaretos19} 
 Let $\mathbb{P},\mathbb{Q}\in\mathcal{D}$. % and $\delta_\Xi$ be the diameter of $\Xi$.
For $n_1 \leq n_2$, then 
\begin{equation}\label{bound_ln_wasserstein}
\begin{aligned}
  & \mathbb{W}_{n_1}(\mathbb{P}, \mathbb{Q}) \leq \mathbb{W}_{n_2}(\mathbb{P}, \mathbb{Q}) \\ & \text{and} \qquad \mathbb{W}_{n_2}(\mathbb{P}, \mathbb{Q})^{n_2} \leq  \mathbb{W}_{n_1}(\mathbb{P}, \mathbb{Q})^{n_1}(2M_\Xi)^{n_2-n_1}.   
\end{aligned}
\end{equation}

(ii)\cite[Lemma 2.10]{Pflug14} Let $\mathbb{P}, \mathbb{Q}_0,\mathbb{Q}_1\in\mathcal{D}$, then for $0 \leq \lambda \leq 1$, it holds that
\begin{equation}\label{convexity_ln_wasserstein}
    \mathbb{W}_n\left(\mathbb{P},(1-\lambda) \mathbb{Q}_0+\lambda \mathbb{Q}_1\right)^n \leq(1-\lambda) \mathbb{W}_n\left(\mathbb{P}, \mathbb{Q}_0\right)^n+\lambda \mathbb{W}_n\left(\mathbb{P}, \mathbb{Q}_1\right)^n.
\end{equation}
 \end{lemma}
Recall also that the $l_1-$Wasserstein metric is equivalent to the Kantorovich metric \eqref{define_distanceP}.
\begin{theorem}\label{duality}(Kantorovich-Rubenstein Theorem)\cite{Edwards11}
The $l_1$-Wasserstein metric between $\mathbb{P},\mathbb{Q}\in \mathcal{D}$,  
$\mathbb{W}_1(\mathbb{Q},\mathbb{P})=\inf_\pi \iint_{\Xi\times\Xi}\left\|\xi-\xi'\right\| \pi\left(d \xi, d \xi'\right),$
can be equivalently written as
$$\rho(\mathbb{Q},\mathbb{P})=\sup _{g \in \mathcal{G}}\left|\mathbb{E}_{\mathbb{Q}}[g(\xi)]-\mathbb{E}_{\mathbb{P}}[g(\xi)]\right|,$$
where $\mathcal{G}$ is a family of Lipschitz continuous functions with constant 1. 
\end{theorem}
We also use the following result from Plfug and Pichler \cite{Pflug14}. It establishes the convergence of the Hausdorff metric between the given probability measure and its discrete approximation. This result was also used by \cite[Proposition 7]{Liu19} in their convergence analysis.
\begin{lemma} (Convergence of Hausdorff Distance between Distributions)\cite[Lemma 4.9]{Pflug14}\label{lemma:distance}
Given a $\mathbb{P} \in \mathcal{P}$, let $\hat{\mathbb{P}}=\left(p_1,\dots,p_{|\Omega|}\right)$ be a solution of \eqref{approximate_probability}: 
%Let $\mathbb{P}\in\mathcal{P}$. A solution to the  following optimization problem exists:
\begin{equation}\label{approximate_probability}
\begin{aligned}
%\mathbb{P}_{|\Omega|} =&\arg
&\min_{\mathbb{Q}\in \mathcal{D}}\;\; \rho(\mathbb{Q},\mathbb{P})\quad  %\quad \text{and} \quad \mathbb{P}_{|\Omega|}\in \mathcal{P}_{|\Omega|}\quad \mathrm{d}(\mathbb{P}_{|\Omega|},\mathcal{P}_{|\Omega|})\leq C\beta_{|\Omega|}.
\text{s.t.} \;\sum_{\omega=1}^{|\Omega|}p_\omega=1, \; p_\omega\geq 0,\quad \forall \omega.
\end{aligned}
\end{equation}
Then
  $\rho\left(\hat{\mathbb{P}}, \mathbb{P}\right)\leq \beta_{|\Omega|}.$
\end{lemma}
 Liu, Yuan, and Zhang \cite[Theorem 3.4]{Liu21} showed that $\mathrm{H}\left(\mathcal{P}_{|\Omega|}, \mathcal{P}\right) \leq 2\beta_{|\Omega|}$ for the case of $n=1$. The following theorem shows the convergence of $\mathcal{P}_{|\Omega|}$ to $\mathcal{P}$ for any $n$.
\begin{theorem}
Let $\mathcal{P}$ and $\mathcal{P}_{|\Omega|}$ be defined as in \eqref{ambigutyset_distance}. Then for sufficiently large ${|\Omega|}$, 
$\mathrm{H}\left(\mathcal{P}_{|\Omega|}, \mathcal{P}\right) \leq 2\beta_{|\Omega|}$.
\end{theorem}
\begin{proof}
Since $\mathcal{P}_{|\Omega|}\subseteq \mathcal{P}$, we have that $\max_{\mathbb{P}_{|\Omega|}\in\mathcal{P}_{|\Omega|}}\mathrm{d}\left(\mathbb{P}_{|\Omega|}, \mathcal{P}\right)=0$. We construct a new set $\mathcal{P}^{(1)}$, where 
\begin{equation}\label{new_w1_set}
\begin{aligned}
\mathcal{P}^{(1)}&:=\{\mathbb{P}\mid \rho(\mathbb{P},\mathbb{P}') \leq d_1\}\\
  & = \{\mathbb{P}\mid (2M_\Xi)^{\frac{n-1}{n}}\rho(\mathbb{P},\mathbb{P}')^{\frac{1}{n}} \leq (2M_\Xi)^{\frac{n-1}{n}}d_1^{\frac{1}{n}}\}\\  
   & = \{\mathbb{P}\mid \mathbb{W}_n(\mathbb{P},\mathbb{P}')\leq (2M_\Xi)^{\frac{n-1}{n}}\rho(\mathbb{P},\mathbb{P}')^{\frac{1}{n}} \leq (2M_\Xi)^{\frac{n-1}{n}}d_1^{\frac{1}{n}}\}. \\ 
   %&\qquad \qquad \qquad\qquad\qquad\qquad\qquad\qquad\qquad \text{(Using \eqref{bound_ln_wasserstein})}  
\end{aligned}
\end{equation}
%For any $\mathbb{P}\in\mathcal{P}^{(1)}$, $$\mathbb{W}_n(\mathbb{P},\mathbb{P}')^n\leq (2M_\Xi)^{n-1}\rho(\mathbb{P},\mathbb{P}')\leq (2M_\Xi)^{n-1}\max_{\mathbb{P}\in\mathcal{P}^{(1)}}\rho(\mathbb{P},\mathbb{P}')\leq (2M_\Xi)^{n-1}d_1\leq d_0^n,$$
%which means that $\mathbb{P}\in\mathcal{P}$ and $\mathcal{P}^{(1)}\subseteq\mathcal{P}$. 
%Also, for any $\mathbb{P}\in\mathcal{P}$, $$\rho(\mathbb{P},\mathbb{P}')\leq \max_{\mathbb{P}\in\mathcal{P}}\rho(\mathbb{P},\mathbb{P}')=d_1,$$
%which means that $\mathbb{P}\in\mathcal{P}^{(1)}$ and $\mathcal{P}\subseteq\mathcal{P}^{(1)}$. 
Let $d_1=\frac{d_0^n}{(2M_\Xi)^{n-1}}$ and then  $\mathcal{P}=\mathcal{P}^{(1)}$. 
%Similarly, we define $$\mathcal{P}^{(1)}_{|\Omega|}:=\{\mathbb{P}_{|\Omega|}\mid \rho(\mathbb{P}_{|\Omega|},\mathbb{P}')\leq d_1, \mathbb{P}_{|\Omega|}=\left(p_1,\dots,p_{|\Omega|}\right),\sum_{\omega=1}^{|\Omega|}p_\omega=1, p_\omega\geq 0\; \forall \omega  \}$$ and then $\mathcal{P}_{|\Omega|}=\mathcal{P}^{(1)}_{|\Omega|}$.
According to Lemma~\ref{lemma:distance}, for any $\mathbb{P}\in\mathcal{P}$, the solution of  \eqref{approximate_probability}, $\hat{\mathbb{P}}$, is a discrete probability measure and $\rho(\mathbb{P},\hat{\mathbb{P}})\leq \beta_{|\Omega|}$.
Then by triangle inequality and \eqref{new_w1_set}, 
\begin{equation}\label{hatp_p_distance}
    \rho(\hat{\mathbb{P}},\mathbb{P}')\leq \rho(\hat{\mathbb{P}}, \mathbb{P}) + \rho(\mathbb{P},\mathbb{P}') \leq \beta_{|\Omega|} + d_1.
\end{equation}
Let $\mathbb{P}_c=\frac{d_1}{\beta_{|\Omega|} + d_1}\hat{\mathbb{P}}+\frac{\beta_{|\Omega|}}{\beta_{|\Omega|} + d_1}\mathbb{P}'$. 
From Lemma~\ref{properties_ln_wasserstein}, 
\begin{equation}
\begin{aligned}
    &\mathbb{W}_n(\mathbb{P}_c, \mathbb{P}')^n \\ \leq &\;  \frac{\beta_{|\Omega|}}{\beta_{|\Omega|} + d_1}\mathbb{W}_n(\mathbb{P}', \mathbb{P}')^n+\frac{d_1}{\beta_{|\Omega|} + d_1}\mathbb{W}_n(\hat{\mathbb{P}}, \mathbb{P}')^n  \quad \text{(Using \eqref{convexity_ln_wasserstein})}\\
     \leq &\;  \frac{d_1}{\rho(\hat{\mathbb{P}},\mathbb{P}')}\mathbb{W}_n(\hat{\mathbb{P}}, \mathbb{P}')^n\qquad \quad \text{(Using \eqref{hatp_p_distance} and $\mathbb{W}_n(\mathbb{P}', \mathbb{P}')=0$)}\\
  \leq &\; \frac{d_1}{\rho(\hat{\mathbb{P}},\mathbb{P}')}\rho(\hat{\mathbb{P}}, \mathbb{P}')(2M_\Xi)^{n-1}  \qquad \qquad \qquad \qquad \text{(Using \eqref{bound_ln_wasserstein})}\\
   % &\leq \;\; d_0 && \text{($\delta_\Xi\leq 1$)}\\
    =&\; d_0^n. \qquad \qquad \qquad\qquad\qquad \qquad \qquad \quad \text{($d_1=\frac{d_0^n}{(2M_\Xi)^{n-1}}$)}\\
    \end{aligned}\nonumber
\end{equation}
Therefore, $ \mathbb{W}_n(\mathbb{P}_c, \mathbb{P}')\leq d_0$, which implies that $\mathbb{P}_c\in \mathcal{P}_{|\Omega|}$. Similarly, we have 
\begin{equation}
    \rho(\mathbb{P}_c,\hat{\mathbb{P}})\leq \frac{\beta_{|\Omega|}}{\beta_{|\Omega|} + d_1}\rho(\hat{\mathbb{P}}, \mathbb{P}')+\frac{d_1}{\beta_{|\Omega|} + d_1}\rho(\hat{\mathbb{P}}, \hat{\mathbb{P}})\leq \beta_{|\Omega|}.\nonumber
\end{equation}
Then for all $\mathbb{P}\in\mathcal{P}$, $\mathrm{d}\left(\mathbb{P}, \mathcal{P}_{|\Omega|}\right)\leq  \rho\left(\mathbb{P}, \mathbb{P}_c\right)\leq \rho\left(\mathbb{P}, \hat{\mathbb{P}}\right) + \rho\left(\mathbb{P}_c,\hat{\mathbb{P}}\right)\leq 2\beta_{|\Omega|}$. %As $\beta_{|\Omega|}\rightarrow 0$ when ${|\Omega|}\rightarrow 0$, 
Clearly,  $\max_{\mathbb{P}\in\mathcal{P}}\mathrm{d}\left(\mathbb{P}, \mathcal{P}_{|\Omega|}\right)\leq 2\beta_{|\Omega|}$. 
Therefore, from \eqref{define_distanceSet}, 
\begin{equation}\nonumber
    \begin{aligned}
  &\mathrm{H}\left(\mathcal{P}_{|\Omega|}, \mathcal{P}\right)
  =\max \left\{\max_{\mathbb{P}_{|\Omega|}\in\mathcal{P}_{|\Omega|}} \mathrm{d}\left(\mathbb{P}_{|\Omega|}, \mathcal{P}\right), \max_{\mathbb{P}\in\mathcal{P}}  \mathrm{d}\left(\mathbb{P}, \mathcal{P}_{|\Omega|}\right)\right\}\leq 2\beta_{|\Omega|}.     
    \end{aligned}
\end{equation}
$~\Box$
\end{proof}

%\begin{theorem}\cite[Theorem 3.4]{Liu21}
%Let $\mathcal{P}$ and $\mathcal{P}_{|\Omega|}$ be defined as in \eqref{ambigutyset_distance}.Then for ${|\Omega|}$ sufficiently large,  $\mathrm{H}\left(\mathcal{P}_{|\Omega|}, \mathcal{P}\right) \leq 2\beta_{|\Omega|}$.
%\end{theorem}
 \bibliographystyle{elsarticle-num} 
 \bibliography{cas-refs}





% \end{thebibliography}
\end{document}